\documentclass[11pt, reqno]{amsart}
\usepackage[margin=1in]{geometry}
\usepackage{amsmath,amssymb}
\usepackage{xcolor} % A package to add color.
\usepackage{amsthm}
\usepackage{amsfonts}
\usepackage{subcaption}
\usepackage{graphicx}
\usepackage[dvpis]{epsfig}
 \usepackage[colorlinks=true]{hyperref}
 \hypersetup{urlcolor=blue, citecolor=red}
\makeatletter
\newcommand*{\rom}[1]{\expandafter\@slowromancap\romannumeral #1@}
\makeatother

\newcommand{\R}{\mathbb{R}}

\newtheorem{theorem}{Theorem}[section]

\newtheorem{lemma}[theorem]{Lemma}

\newtheorem{proposition}[theorem]{Proposition}

\theoremstyle{definition}

\theoremstyle{remark}

\newtheorem{remark}[theorem]{Remark}

\numberwithin{equation}{section}

\newcommand{\T} {\mathbb T}
\newcommand{\pa}{\partial}

\newcommand{\lt}{\left}
\newcommand{\rt}{\right}
\newcommand{\bq}{\begin{equation}}
\newcommand{\eq}{\end{equation}}

\newcommand{\calC}{\mathcal C}

\newcommand{\calE}{\mathcal E}

\newcommand{\calL}{\mathcal L}

\newcommand{\calR}{\mathcal R}

\newcommand{\calX}{\mathcal X}

\newcommand{\intt}{\int_\T}

\renewcommand{\geq}{\geqslant}
\renewcommand{\ge}{\geqslant}
\renewcommand{\leq}{\leqslant}
\renewcommand{\le}{\leqslant}
\newcommand{\rd}{\textnormal{d}}
\newcommand{\dx}{\textnormal{d}x}
\newcommand{\dy}{\textnormal{d}y}
\newcommand{\dz}{\textnormal{d}z}
\newcommand{\dt}{\textnormal{d}t}

\usepackage{soul}
\usepackage[normalem]{ulem}
\usepackage{cancel}

\renewcommand{\sout}[1]{}
\renewcommand{\cancel}[1]{}
\begin{document}

\title[Large-time behavior in pressureless Euler--Poisson equations]{Large-time behavior of  pressureless\\ Euler--Poisson equations with background states}

\author{Young-Pil Choi}
\address{Department of Mathematics, Yonsei University, Seoul 03722, Republic of Korea}
\email{ypchoi@yonsei.ac.kr}

\author{Dong-ha Kim}
\address{CAU Nonlinear PDE Center, Chung-ang University, Seoul 06974, Republic of Korea}
\email{kimdongha91@cau.ac.kr}

\author{Dowan Koo}
\address{Department of Mathematics, Yonsei University, Seoul 03722, Republic of Korea}
\email{dowan.koo@yonsei.ac.kr}

\author{Eitan Tadmor}
\address{Department of Mathematics and IPST, University of Maryland, College Park, MD 20742, USA}
\email{tadmor@umd.edu}

\date{\today}


\thanks{\textbf{Acknowledgment.} 
The research of YPC and DK were supported by  NRF grant 2022R1A2C1002820 and RS-2024-00406821. The research of ET was supported by ONR grant N00014-2412659, NSF grant DMS-2508407 and  by the Fondation Sciences Math\'ematiques de Paris (FSMP) while being hosted by
LJLL at Sorbonne University.}

\begin{abstract}
We study the large-time asymptotic behavior of solutions to the one-dimensional damped pressureless Euler--Poisson system with variable background states, subject to a neutrality condition. In the case where the background density converges asymptotically to a positive constant, we establish the convergence of global classical solutions toward the corresponding equilibrium state. The proof combines phase plane analysis with hypocoercivity-type estimates. As an application, we analyze the damped pressureless Euler--Poisson system arising in cold plasma ion dynamics, where the electron density is modeled by a Maxwell--Boltzmann relation. We show that solutions converge exponentially to the steady state under suitable a priori bounds on the density and velocity fields. Our results provide a rigorous characterization of asymptotic stability for damped Euler--Poisson systems with nontrivial background structures.
\end{abstract}

% ----------------------------------------------------------------

\maketitle

% ----------------------------------------------------------------
\tableofcontents
%%%%%%%%%%%%%%%%%%%%%%%%%%%%%%%%%%%%%%%%%%%%%%%%%%%%%%%%%
%
%
%
%
%
%%%%%%%%%%%%%%%%%%%%%%%%%%%%%%%%%%%%%%%%%%%%%%%%%%%%%%%%%

\section{Introduction}
In this paper, we consider the damped pressureless Euler--Poisson (EP) system with background states in one spatial dimension and analyze the large-time behavior of its classical solutions. Specifically, we study the global-in-time dynamics of solutions to the following system posed on the one-dimensional periodic domain $\T= [-\frac{1}{2},\frac{1}{2})$:
\begin{equation}\label{main_sys}
\begin{cases}
\partial_t \rho + \partial_x(\rho u) = 0, \\
\partial_t u + u \partial_x u = -\nu u  -k \pa_x\phi,\\
-\pa^2_{x} \phi = \rho - c,
\end{cases} %\qquad (t,x)\in \R_+ \times \T,
\end{equation}
subject to initial data
\begin{equation}\label{ini}
(\rho, u)(0,x)=(\rho_0, u_0)(x), \quad x \in \T.
\end{equation}
Here, $\rho=\rho(t,x)$ and $u=u(t,x)$ denote the density and velocity fields, respectively, and $\pa_x\phi = \pa_x\phi(t,x)$ represents the induced force. The parameter $\nu > 0$ is the strength of damping, and the parameter $k$ on the right of  \eqref{main_sys}${}_2$ is a physical constant which characterizes the forcing of the system ---  either an \emph{attractive} or \emph{repulsive} forcing depending on whether  $k <0$ or, respectively, $k>0$.

The Poisson equation \eqref{main_sys}${}_3$ involves a prescribed positive background state, $c = c(t,x)$, which is assumed to be uniformly bounded away from vacuum:
\bq \label{eq:cbd}
 0 <  c_{-} \leq c(t,x) \leq c_{+}, \quad \forall \, (t,x) \in \mathbb{R}_+ \times \T,
\eq
where $c_-$ and $c_+$ are positive constants. We further impose the \emph{neutrality condition}:
\bq\label{def_neu} 
\int_{\T} \rho(t,x) - c(t,x) \,\dx = 0 , \quad \forall \, t \ge0,
\eq
which ensures that the total deviation of the density from the background is zero at each time.

Throughout the paper, a triple $(\rho, u, \phi)$ is called a {\it global classical solution} to the system  \eqref{main_sys}--\eqref{def_neu} if it satisfies the system pointwise for all $t>0$, subject to the initial condition \eqref{ini} and the neutrality condition \eqref{def_neu}.

 The EP system serves as a fundamental model in various physical contexts, including plasma physics, semiconductor theory, and self-gravitating fluids. In particular, the damped EP system with variable background states arises naturally in the modeling of carrier transport in semiconductors  \cite{J01, MRS90}. In such contexts, the background profile $c(t,x)$ is referred to as the \emph{doping profile}, representing the distribution of fixed charges inside the material. The damping term models relaxation effects such as scattering with the lattice. The interplay between the carrier density $\rho$ and the doping profile $c$, mediated through the Poisson equation, governs the electric potential $\phi$ and thereby the carrier dynamics. From both modeling and analytical perspectives, understanding the large-time behavior of solutions to damped EP systems with background states provides valuable insights into equilibration mechanisms and long-time stability.

The mathematical analysis of EP systems has been concerned with several important directions, notably the study of global existence versus finite-time singularity formation and the investigation of the large-time asymptotic behavior of global solutions. In particular, the notion of \emph{critical thresholds} was introduced in \cite{ELT01} for the one-dimensional pressureless EP system, providing sharp conditions for global smooth solutions versus finite-time breakdown.  This approach relied on a detailed analysis of characteristic flows and has been further developed primarily in the one-dimensional setting. Some restricted extensions to multidimensional cases were achieved under symmetry assumptions \cite{LT02,LT03}.

The influence of damping, background profiles, and pressure has also been extensively studied. Convergence to equilibrium under suitable subcritical conditions has been shown in \cite{CCZ16, CCTT16, TT14, TW08}, while sharp critical thresholds in the pressureless case with variable background states have been identified using phase plane analysis \cite{BL20_2, CKKTpre}. The global-in-time regularity of one-dimensional solutions has been established in \cite{GHZ17}. In higher dimensions, global existence and stability results for irrotational flows with constant background were obtained in \cite{GP11, IP13, JLZ14, LW14}. Nonlocal variants such as the damped pressureless Euler--Riesz system were analyzed in \cite{CJ23}, where algebraic and exponential decay rates were obtained using hypocoercivity arguments. Global convergence to equilibrium for ion dynamics with Maxwell--Boltzmann-type electron density was studied in \cite{LP19} through refined energy estimates. We also refer to \cite{BCK24, CKKTpre, L06} for blow-up results in cold plasma ion models, and to \cite{TW08} for the critical thresholds estimates in the pressured case.

In our earlier work \cite{CKKTpre}, motivated by \cite{BL20}, we provided a systematic study of the critical threshold phenomena for pressureless EP systems with variable background states under the neutrality condition \eqref{def_neu}. We established local-in-time well-posedness in function spaces involving negative Sobolev regularity, necessitated by the infinite mass structure induced by neutrality. A key observation was that the presence of damping ($\nu>0$) plays a crucial role in obtaining global regularity for the repulsive case ($k>0$), leading to subcritical conditions guaranteeing global existence.

In contrast, for the attractive case ($k<0$), we showed that the neutrality condition severely restricts the class of admissible backgrounds, effectively reducing to the constant background case in order to maintain global regularity (see \cite[Theorem 1.5]{CKKTpre}). In this reduced setting, global-in-time solutions converge exponentially to equilibrium by direct arguments.
 
Motivated by these observations, the present paper focuses on the repulsive case ($k>0$), and for simplicity, we normalize $k=1$. Our main objective is to characterize the precise large-time asymptotic behavior of global classical solutions under this setting.
%%%%%%%%%%%%%%%%%%%%%%%%%%%%%%%%%%%%%%%%%%%%%%%%%%%%%%
%
%
%
%
%
%
%
%
%%%%%%%%%%%%%%%%%%%%%%%%%%%%%%%%%%%%%%%%%%%%%%%%%%%%%% 
\subsection{Exponential convergence for solutions with asymptotically constant backgrounds}
For the case of constant background $c \equiv \bar{c}$, it is known that an explicit representation formula for solutions to \eqref{main_sys} can be obtained via the method of characteristics (see \cite{ELT01}). This explicit formula provides sharp exponential decay rates of global classical solutions toward equilibrium as $t \to \infty$.

When the background $c = c(t,x)$ varies in time and space, but converges asymptotically to a constant state, it is natural to expect a similar convergence behavior for the solutions. In particular, we assume that
\bq\label{eq:cdcy}
c(t,x) \to \bar c\qquad \text{as} \,\,\, t \to \infty
\eq
for some $\bar c>0$ in a suitable sense, and we investigate the large-time behavior of global classical solutions to \eqref{main_sys}. 

 A key novelty of our result lies in the fact that {\it no smallness condition} on the initial data or on the perturbation from equilibrium is required.   In contrast to earlier works that establish global existence and exponential convergence for small perturbations around constant steady states (e.g., via energy methods or linearization techniques), our analysis handles general (large) global classical solutions satisfying uniform-in-time bounds \eqref{eq:asm}. This robustness is achieved by combining phase plane analysis with hypocoercivity-type estimates, which allow us to control time-dependent perturbations introduced by the varying background state.

We now state our first main theorem:

\begin{theorem}\label{thm:general}
    Let $(\rho, u)$ be a global classical solution of \eqref{main_sys} satisfying
\bq\label{eq:asm}
0 <   {\rho}_- \le \rho(t,x) \le {\rho}_+ < + \infty,\quad \|\pa_x u\|_{L^\infty(\T\times\R_+)} \le M <+\infty
\eq
for some $\rho_-, \rho_+,M >0$.
\begin{enumerate}
    \item If $c(t,x)$ satisfies
    \begin{equation}\label{eq:c1}
         \lim_{t\to\infty}\|c(t,\cdot)-\bar{c}\|_{L^{\infty}(\T)}=0,
    \end{equation}
    then we have 
\begin{equation*}
    \lim_{t\to\infty} \|(\rho-\bar{c}\, ,u,\,\partial_x u, \pa_x \phi, \pa^2_x \phi)\|_{L^{\infty}(\T)} =0.
\end{equation*}

\item If we further assume that
\begin{equation}\label{eq:c2}
\|c(t,\cdot)-\bar{c}\|_{L^{\infty}(\T)}\le C_1e^{-r_1t},
\end{equation}
for some constants $r_1>0$ and $C_1 > 0$, then we have
\begin{equation*}
     \|(\rho-\bar{c}\,,\,u,\,\partial_x u\,,\,\pa_x\phi,\pa^2_x\phi)\|_{L^{\infty}(\T)} \le C_2e^{-r_2t}
\end{equation*}
for some $C_2=C_2(C_1,\rho_\pm,,M,c_\pm, \|u_0\|_{L^1(\T)})>0$ and  $r_2=r_2(r_1, \nu, \bar c)>0$.
\end{enumerate}
\end{theorem}

\begin{remark}
When the spatial domain is the real line $\R$ instead of the periodic torus $\T$, similar decay estimates can be established for the quantities
\[
\rho -\bar c,\, \pa_xu,\, \pa^2_x \phi
\]
by relying on the method of characteristics. The proof in this setting does not rely on the boundedness of the domain. However, the convergence of additional quantities such as $u$ and $\pa_x \phi$ crucially uses the boundedness of the domain and the associated Poincar\'e inequality. In particular, the $\|u_0\|_{L^1(\T)}$ dependence in $C_2$ is only needed for the convergence of zeroth order term $\|u\|_{L^\infty(\T)}$.
\end{remark}

\begin{remark} A related exponential decay result to Theorem \ref{thm:general} was established in \cite{CJ21}, which studied the large-time behavior of solutions to a damped pressureless Euler--Poisson system coupled with an incompressible Navier--Stokes system on the torus $\T^d$ for $d \geq 2$, under the assumption of constant background states. In comparison, our result applies to time-dependent background profiles and requires weaker regularity assumptions on the density: while \cite{CJ21} assumes $\rho \in W^{1,\infty}$, we only require $\rho \in L^\infty$. Moreover, our result provides exponential convergence of all relevant quantities in the $L^\infty$ norm, whereas \cite{CJ21} establishes convergence in the $L^2$ framework.
\end{remark}

%%%%%%%%%%%%%%%%%%%%%%%%%%%%%%%%%%%%%%%%%%%%%%%%%%%%%%
%
%
%
%
%
%
%
%
%%%%%%%%%%%%%%%%%%%%%%%%%%%%%%%%%%%%%%%%%%%%%%%%%%%%%% 
 \subsection{Application to cold plasma ion dynamics}

In many physically relevant settings, Euler--Poisson systems are studied under the assumption that the background density asymptotically approaches a constant state, as formulated in \eqref{eq:cdcy}. Motivated by this, we investigate the large-time behavior of solutions to a specific damped pressureless Euler--Poisson system arising in cold plasma ion dynamics:
\begin{equation}\label{eq:ion}
    \begin{cases}
\rho_t + \pa_x (\rho u) = 0, \\
u_t + u \pa_x u =  -\nu u  -\pa_x\phi, \\
- \pa_{xx}\phi = \rho-e^\phi,
\end{cases} \qquad (t,x)\in \R_+ \times \T,
\end{equation}
where $\rho$ and $u$ represent the density and velocity of ions, respectively.  

This model describes a plasma composed of \emph{massless} electrons and ions with constant temperatures (one for electrons, zero for ions). Here, the electron density is given by the \emph{Maxwell--Boltzmann relation} $e^\phi$, with $\phi = \phi(t,x)$ denoting the electric potential generated by the overall charge distribution (see  \cite{C84}). The equation $\eqref{eq:ion}_3$ is often referred to as the \emph{Poisson--Boltzmann equation}.

It is clear that the Poisson--Boltzmann equation satisfies a neutrality condition:
\[
\intt \rho - e^\phi\,\dx =0.
\]
For simplicity, we assume that the initial ion density $\rho$ has unit mass, so that 
\bq\label{eq:mb:neu}
\intt e^\phi -1\,\dx = 0.
\eq
This model can be interpreted as a special case of the general system \eqref{main_sys} with $(k, c) = (1, e^\phi)$. In order to apply Theorem \ref{thm:general} to this system, it is necessary to show that the background density $c(t,x) = e^{\phi(t,x)}$ converges to a constant state as $t \to \infty$ in a suitable sense. The analysis of this asymptotic behavior forms the core of our second main result. 

Importantly, the following theorem demonstrates that the exponential convergence to equilibrium holds without any smallness assumption on the initial perturbation. This goes beyond the scope of previous results on plasma models, which typically assume that the initial data is close to equilibrium in a strong norm.
\begin{theorem}\label{thm:ion}
Let $(\rho, u)$ be a global classical solution of \eqref{eq:ion} satisfying
\bq\label{eq:asm2}
0 <   \rho_- \le \rho(t,x) \le \rho_+ < + \infty,\quad \|\pa_x u\|_{L^\infty(\T\times\R_+)} \le M<+\infty
\eq
for some $\rho_-, \rho_+, M >0$.
Then, we have 
\[
\|(\rho-1\,,\, u\,,\, \pa_x u\,,\, \pa_x\phi, \pa^2_x\phi, e^\phi -1)\|_{L^\infty(\T)} \leq Ce^{-rt}
\]
for some $C,r>0$ depending only on $\nu, \rho_\pm,M, \calE(0)$. 
\end{theorem}

 \begin{remark} Theorem \ref{thm:ion} is stated under an {\it a priori} assumption that a global classical solution satisfying the conditions in \eqref{eq:asm2} exists. However, in \cite{CKKTpre}, the subcritical region of initial data leading to the global existence and uniqueness of classical solutions was analyzed. In particular, initial data satisfying these subcritical conditions guarantee the global-in-time existence of solutions that fulfill the assumptions of Theorem \ref{thm:ion}, thereby validating the large-time behavior stated therein. 
 \end{remark}

 The remainder of the paper is organized as follows. In Section \ref{sec_pphe}, we prove Theorem \ref{thm:general} by combining phase plane analysis with hypocoercivity-type estimates, establishing exponential convergence of solutions under asymptotically constant backgrounds. In Section \ref{sec_cid}, we apply these results to the cold plasma ion dynamics model, and prove Theorem \ref{thm:ion}, demonstrating exponential convergence of solutions to the steady state.

%%%%%%%%%%%%%%%%%%%%%%%%%%%%%%%%%%%%%%%%%%%%%%%%%%%%%%
%
%
%
%
%
%
%
%
%%%%%%%%%%%%%%%%%%%%%%%%%%%%%%%%%%%%%%%%%%%%%%%%%%%%%% 
\section{Phase plane analysis and hypocoercivity estimate}\label{sec_pphe}

In this section, we introduce a phase plane formulation and derive hypocoercivity estimates to study the large-time behavior of solutions. To better capture the nonlinear transport effects and the influence of the background forcing, we introduce rescaled Lagrangian variables that transform the EP system into a more tractable ODE system along characteristics. Following the approach introduced by H. Liu, S. Engelberg and the last author in \cite{ELT01}, we define the Lagrangian variables 
\bq\label{eq:swc}
s:= \frac{1}{\rho}, \quad w:=\frac{\pa_x u}{\rho}
\eq
along the characteristic 
\[
x'(t)=u(t,x(t)). 
\]
For simplicity, we slightly abuse notation and denote $w(t):= w(t,x(t))$ and $s(t):= s(t,x(t))$.

Under these variables, the system \eqref{main_sys} reduces to the following ODE system:
\begin{equation}\label{eq:1sys}
\lt\{\begin{split}
w' &= -\nu w + 1 - c s, \\
s' &= w. 
\end{split}\rt.
\end{equation}
The global-in-time regularity of solutions is guaranteed if $s(t)$ remains positive for all time, thereby preventing finite-time blow-up of $\rho(t,\cdot)$. We recall from \cite[Theorem 1.9]{CKKTpre} that the subcritical region --- i.e.  the set of initial data ensuring global-in-time regularity--- is characterized therein.

To study the large-time behavior of $(w(t),s(t))$, we begin with the following auxiliary lemma.
\begin{lemma}\label{lem:phase}
Let $(w(t),s(t))$ be a global classical solution to \eqref{eq:1sys} satisfying
\bq\label{eq:swB}
|s|,|w|\le B
\eq
for some $B>0$. Assume that the background $c(t)$ satisfies 
\[
0< c_- \le c(t) \le c_+<+\infty
\] 
and 
\bq\label{eq:l:c1}
 |c(t) - \bar c| \le g(t)
\eq
for some nonnegative $g \in C([0,\infty))$ satisfying $g(t) \to 0$ as $t \to \infty$. Then there exist constants $C_0 =C_0(\nu, B, c_\pm)>0$ and $r_0=r_0(\nu,\bar c)>0$ such that 
\bq\label{eq:dcy1}
\lt|s(t) - \frac{1}{\bar c} \rt|^2 + |w(t)|^2 \le C_0 \lt(e^{-r_0t} + \sup_{\tau \in [t/2, t]} g(\tau) \rt) \to 0 \quad\text{as}\quad t \to \infty.
\eq
If we further assume  
\bq\label{eq:l:c2}
|c(t) -\bar c| \le C_1e^{-r_1 t},
\eq
for some $r_1 > 0$, then we have
\bq\label{eq:dcy2}
\lt|s(t)-\frac{1}{\bar{c}}\rt| + |w(t)| \le C_2e^{-r_2t}
\eq
for some $C_2=C_2(  \nu, C_1,B,c_\pm)>0$ and  $r_2=r_2(r_1, \nu, \bar c)>0$. 
\end{lemma}
\begin{proof}
We define the energy functional 
\[
\calL(t):= \frac{\bar{c}}{2}\lt(s(t)-\frac{1}{\bar c}\rt)^2 + \frac12 w(t)^2.
\]
A direct computation using \eqref{eq:1sys} and \eqref{eq:swB} yields
\[
\calL'(t) = -\nu w^2 - (c-\bar{c})sw \le -\nu w^2 + |c-\bar c|B^2 .
\]
To reveal the hidden dissipation structure and control the coupling between $s$ and $w$, we introduce the cross term
\[
\calX(t):=\lt(s(t)-\frac{1}{\bar c}\rt)w(t). 
\]
This cross term is crucial for establishing hypocoercivity estimates in the system.

Its temporal derivative is computed as
\[
\begin{split}
\calX' =&\,w^2 + \lt(s-\frac{1}{\bar{c}}\rt)(-\nu w + 1 - c s) \\
=&\,- \bar c \lt(s-\frac{1}{\bar c}\rt)^2 + w^2 -\nu \lt(s-\frac{1}{\bar c}\rt)w -(c-\bar c)s\lt(s-\frac{1}{\bar c}\rt) \\
=:&\,-\bar c  \lt(s-\frac{1}{\bar c}\rt)^2 + \calR,
\end{split}
\]
where the remainder term satisfies
\[
|\calR| \le \frac{\bar c }{2}\lt(s-\frac{1}{\bar c}\rt)^2 +\lt(1+\frac{\nu^2}{2\bar c }\rt)w^2 + |c-\bar c|B\lt(B+\frac{1}{\bar c}\rt).
\]
Hence, we obtain the differential inequality
\bq \label{eq1346fri}
(\calL +\lambda \calX)' \leq -\lt(\nu - \lambda\lt(1+\frac{\nu^2}{2\bar c }\rt)\rt) w^2 - \frac{\lambda \bar c }{2} \lt(s-\frac{1}{\bar c}\rt)^2 + |c-\bar c|B\lt(B+\lambda\lt(B+\frac{1}{\bar c}\rt)\rt)
\eq
for a small parameter $\lambda > 0$. Choosing $\lambda$ sufficiently small so that
\bq \label{lam_bound1}
\nu - \lambda\lt(1+\frac{\nu^2}{2\bar c }\rt)\geq \frac{\lambda}{2},
\eq
and setting
\begin{equation*}
    \begin{aligned}
        N:=B\lt(B+\lambda\lt(B+\frac{1}{\bar c}\rt)\rt),
    \end{aligned}
\end{equation*}%for notational convenience. 
then we reduce from \eqref{eq1346fri} that
\bq\label{eq:hypo:pre}
    (\calL +\lambda \calX)' \leq  - \lambda \calL + N|c-\bar c|.
\eq
To make $\calL$ and $\calL + \lambda \calX$ be comparable, we observe that %We apply Young's inequality to obtain that
%\begin{equation*}
    $|\calX(t)| \leq \frac{1}{\sqrt{\bar c}}\calL(t)$, 
%\end{equation*}
and hence,
\begin{equation*}
    \lt(1-\frac{\lambda}{\sqrt{\bar c}}\rt) \calL(t) \leq \calL(t) + \lambda \calX (t) \leq    \lt(1+\frac{\lambda}{\sqrt{\bar c}}\rt) \calL(t).
\end{equation*}
In particular, if $\lt(1-\frac{\lambda}{\sqrt{\bar c}}\rt) \geq \frac{1}{2}$ and $\lt(1+\frac{\lambda}{\sqrt{\bar c}}\rt) \leq 2$, then we get
\[
\frac{1}{2} \calL(t) \leq \calL(t) + \lambda \calX (t) \leq  2\calL(t).
\]
Combining this with \eqref{lam_bound1}, we finally fix 
\[
\lambda := \min\lt\{ \nu\lt(\frac{\nu^2}{2\bar c} + \frac{3}{2}\rt)^{-1}, \ \frac{\sqrt{\bar c}}{2} \rt\}>0.
\]
As a consequence, it follows that
\[
\frac{1}{2}\calL(t) \leq (\calL(t)+ \lambda \calX(t))=:y(t),
\]
and \eqref{eq:hypo:pre} becomes
\begin{equation*}
    y'(t) \leq -\frac{\lambda}{2} y(t) +N|c(t)-\bar{c}|. 
\end{equation*}
We use Gr\"onwall's inequality to obtain that
\begin{equation}\label{eq1807sun}
    \begin{aligned}
        y(t) &\leq y(0)e^{-\frac{\lambda}{2} t} +N\int_0^t e^{-\frac{\lambda}{2}(t-\tau)}|c(\tau)-\bar{c}| \,\textnormal{d} \tau\\
        & \leq y(0)e^{-\frac{\lambda}{2} t} +2N\frac{(c_+-c_-)}{\lambda} (e^{-\lambda t/4}-e^{- \lambda t/2}) +\frac{2N}{\lambda}\sup_{\tau \in [t/2, t]} g(\tau).
    \end{aligned}
\end{equation}
This proves \eqref{eq:dcy1} for $c(t)$ satisfying \eqref{eq:l:c1}. On the other hand, if $c(t)$ satisfies 
\eqref{eq:l:c2}, then it follows from the first line of \eqref{eq1807sun} that we obtain the exponential decay rate of convergence, leading to \eqref{eq:dcy2}.
This completes the proof.
\end{proof}
%Hence, it suffices to establish the temporal decay of $y(t)$ as $t \to \infty$.

\subsection{Proof of Theorem \ref{thm:general}}

%We are now ready to prove Theorem \ref{thm:general}.
%\begin{proof}[Proof of Theorem \ref{thm:general}] 
For each $x\in \T$, we define the Lagrangian variables $s=s(t), w=w(t)$ as in \eqref{eq:swc}. 
By the assumptions \eqref{eq:asm}, we have the uniform bounds
\[
\frac{1}{\rho} \le \frac{1}{\rho_-}, \quad \lt|\frac{\pa_x u}{\rho} \rt| \le \frac{M}{\rho_-}
\]
so that \eqref{eq:swB} holds with
\[
B:=\frac{\max\{1,M\}}{\rho_-}.
\]
Moreover, for any $x\in \T$, we observe that
\[
|c(t,x(t,x))-\bar c| \le \|c(t,\cdot)-\bar c\|_{L^\infty(\T)}=:g(t),
\]
thus, the conditions \eqref{eq:l:c1} and \eqref{eq:l:c2} hold uniformly in $x\in \T$ thanks to \eqref{eq:c1} and \eqref{eq:c2}, respectively.

To establish the convergence of $\rho-\bar{c}$, $\pa_xu$, we note that
\begin{equation*}
    |\rho-\bar c| = \lt|\rho \cdot \bar c\left(\frac{1}{\rho} - \frac{1}{\bar c}\right)\rt| \leq \rho_+  \bar c \left|\frac{1}{\rho} - \frac{1}{\bar c}\right|
\end{equation*}
and
\begin{equation*}
   |\pa_xu| = \left|\frac{\pa_x u}{\rho}\right||\rho| \leq \rho_+ \left|\frac{\pa_x u}{\rho}\right|.  
\end{equation*}
Hence, the decay of $s(t) - \frac{1}{\bar{c}}$ and $w(t)$ obtained in Lemma \ref{lem:phase} directly implies the convergence of $\rho-\bar{c}$ and $\partial_x u$ in $L^\infty(\T)$ as $t\to\infty$.

The decay of $\pa^2_x \phi$ can be easily seen as
\[
 \| \pa^2_x\phi\|_{L^\infty(\T)}= \|\rho -c\|_{L^\infty(\T)} \le \|\rho -\bar{c}\|_{L^\infty(\T)} + \|c-\bar{c}\|_{L^\infty(\T)}.
\]
Since both terms decay, we deduce that $\|\partial_x^2 \phi(t)\|_{L^\infty(\T)} \to 0$ as $t\to\infty$.

To prove the convergence of $\pa_x\phi$, we note that 
\[
\intt \pa_x\phi\,\dx = 0.
\]
Thus, by Sobolev and H\"older inequalities, we find that
\[
\|\pa_x\phi\|_{L^\infty(\T)}\le \| \pa^2_x\phi\|_{L^1(\T)} \le \| \pa^2_x\phi\|_{L^\infty(\T)},
\]
which implies the decay of $\|\partial_x \phi\|_{L^\infty(\T)}$.

Finally, we turn to the convergence of $u$. By integrating $\eqref{main_sys}_2$, we have
\[
\frac{\rd}{\dt}\intt u\,\dx = -\nu \intt u\,\dx. 
\]
Setting $m(t):=\intt u \,\dx$, we obtain the explicit formula
\[
m(t) = m(0)e^{-\nu t}.
\]
Hence, by Poincar{\'e} and H{\"o}lder inequalities, we deduce that
\[
\|u\|_{L^2(\T)} \le \|u - m(t)\|_{L^2(\T)} + m(t) \le \|\pa_x u\|_{L^2(\T)} + m(t) \le \|\pa_x u\|_{L^\infty(\T)} +m(t).
\]
To obtain the convergence of $\|u\|_{L^{\infty}(\T)}$, we note that
\[
|u(x)| = \lt|u(y) + \int_y^x \pa_z u(z)\,\dz\rt| \le |u(y)| + \|\pa_x u\|_{L^\infty(\T)}
\]
for any $x, y \in \T$. By integrating in $y$ variable, we obtain
\[
|u(x)| \le \lt(\intt |u(y)|^2\,\dy\rt)^{1/2} + \|\pa_x u\|_{L^\infty(\T)} = \|u\|_{L^2(\T)} + \|\pa_x u\|_{L^\infty(\T)},
\]
proving decay estimate of $u$ in $L^\infty$ norm. This completes the proof.
%\end{proof}

%%%%%%%%%%%%%%%%%%%%%%%%%%%%%%%%%%%%%%%%%%%%%%%%%%%%%%
%
%
%
%
%
%
%
%
%%%%%%%%%%%%%%%%%%%%%%%%%%%%%%%%%%%%%%%%%%%%%%%%%%%%%% 
\section{Application to cold ion dynamics}\label{sec_cid}

We now apply the general framework developed in Section \ref{sec_pphe} to study the large-time behavior of the damped cold plasma model.

\subsection{Energy estimates}

For classical solutions to \eqref{eq:ion}, we introduce the free energy functional
\[
\calE(t):= \int_{\T} \frac{1}{2}\rho u^2 + \frac{1}{2} (\pa_x\phi)^2 +e^\phi(\phi -1) + 1 \,\dx,
\]
which is dissipated over time due to damping. A direct computation yields the energy dissipation law:
\bq\label{eq:edr}
\frac{\textnormal{d}}{\dt}\calE(t) = -\nu \int_{\T} \rho u^2\,\dx.
\eq
Since the linear damping term acts only on the momentum and not directly on the other components of $\mathcal{E}(t)$, we establish exponential decay of the free energy via a hypocoercivity argument.

\begin{proposition}\label{prop:1}
Let $(\rho,u)$ be a global-in-time classical solution to \eqref{eq:ion} satisfying the uniform bounds \eqref{eq:asm} as indicated in Theorem \ref{thm:ion}.  Then there exists a constant $r>0$, depending only on $\nu, \rho_\pm, M$, and $\calE(0)$, such that 
\[
\calE(t) \le 3 \calE(0)e^{-rt}, \quad \forall \,t>0.
\]
\end{proposition}
\begin{proof}
To obtain the hypocoercivity structure, we introduce the cross term
\[
\calC(t) :=  \int_{\T} u \pa_x\phi\, \dx.
\]
We then compute its temporal derivative as
\begin{align*}
\calC'(t) &= -\intt ( u\pa_x u + \nu u ) \pa_x\phi \,\dx - \intt (\pa_x\phi)^2 \,\dx + \intt u (\pa_x\phi_t) \,\dx\cr
&=:I + II + III.
\end{align*}
We observe that $II$ has a favorable (negative) sign, while $I$ and $III$ require careful estimates.

For $I$, we obtain
\[
\begin{split}
|I| &\le (\|\pa_x u\|_{L^\infty(\R_+\times\T)} + \nu) \int_{\T} |u \pa_x \phi| \,\dx \\
&\le \frac{(M + \nu)^2}{ \rho_-}\int_{\T} \frac{1}{2}\rho u^2\,\dx + \int_{\T} \frac{1}{2} (\pa_x \phi)^2\,\dx.
\end{split}
\]

For $III$, we differentiate the Poisson--Boltzmann equation $\eqref{eq:ion}_3$:
\[
\lt(-\pa_{xx} + e^{\phi}\rt)\phi_t= \rho_t = -\pa_x(\rho u).
\]
Testing against $\phi_t$ and integrating by parts and then applying Young's inequality, we find
\[
\begin{split}
\int_{\T} (\pa_x \phi_t)^2 + e^{\phi}(\phi_t)^2\,\dx &=  - \int_{\T} \pa_x(\rho u)\phi_t\,\dx \\
&= \int_{\T}  \rho u(\pa_x \phi_t)\,\dx \\
&\le \int_{\T} \frac{1}{2}(\rho u)^2\,\dx + \int_{\T} \frac{1}{2}  (\pa_x \phi_t)^2\,\dx.
\end{split}
\]
Thus, we estimate
\[
|III| \le \int_{\T} \frac{1}{2}u^2\,\dx + \int_{\T} \frac{1}{2} (\pa_x\phi_t)^2\,\dx \le \int_{\T} \frac{1}{2}u^2\,\dx + \int_{\T} \frac{1}{2} (\rho u)^2\,\dx \le \lt(\frac{1}{\rho_-} + \rho_+ \rt) \int_{\T} \frac{1}{2} \rho u^2\,\dx.
\]
Defining
\[
\Lambda:= \frac{(M + \nu)^2}{ \rho_-} + \lt(\frac{1}{\rho_-} + \rho_+ \rt) \ge 2,
\]
we deduce 
\[
\frac{\textnormal{d}}{\dt} \lt( \calE(t) + \lambda \calC(t) \rt) \le -(2\nu - \Lambda \lambda) \int_{\T} \frac{1}{2}\rho u^2\,\dx - \lambda \int_{\T} \frac{1}{2} (\pa_x \phi)^2\,\dx
\]
for any $\lambda > 0$.

It remains to control the entropy term
\[
\int_{\T} e^\phi(\phi-1)+1\,\dx.
\]
We  observe that the quantity $\|e^\phi -1\|_{L^\infty(\T)}$ can be controlled by free energy. Indeed, thanks to \eqref{eq:mb:neu}, there exists $\bar{x} \in \T$ such that $\phi(\bar{x})=0$. Thus, for any $x\in\T$, we observe that
\bq\label{eq:phi:dcy}
|\phi(x)|= |\phi(x)-\phi(\bar{x})| = \lt|\int_{\bar{x}}^{x} \pa_y \phi(y)\,\dy\rt| \le \lt(\int_\T |\pa_y \phi(y)|^2\,\dy\rt)^{1/2} \le (2\calE(t))^{1/2}.
\eq
In particular, combining \eqref{eq:phi:dcy} with the energy dissipation estimate \eqref{eq:edr}, we have 
\bq\label{eq:phibd}
\|\phi\|_{L^\infty(\T)} \le (2\calE(0))^{1/2} =: A.
\eq
We then interpret $s:=e^\phi$, and apply Taylor's expansion for $s \mapsto s \log s -s +1$ to estimate the entropy term as
\[
\int_{\T} e^{\phi}(\phi-1)+1 \,\dx \le e^{A}\int_{\T} (e^\phi-1)^2 \,\dx.
\]
Due to the neutrality $\int_{\T} e^\phi -1\,\dx =0$ and Poincar{\'e} inequality, we have
\[
\intt (e^{\phi}-1)^2\,\dx \le \frac{1}{4}\intt(\pa_x(e^\phi))^2\,\dx \le  \frac{e^{2A}}{4}\intt (\pa_x \phi)^2\,\dx.
\]
This shows that the entropy term can be controlled by electric energy. Hence, we deduce
\[
\frac{\textnormal{d}}{\dt} \lt( \calE(t) + \lambda \calC(t) \rt) \le -(2\nu - \Lambda \lambda) \int_{\T} \frac{1}{2}\rho u^2\,\dx - \frac{\lambda}{2}\int_{\T} \frac{1}{2} (\pa_x \phi)^2\,\dx - \lambda e^{-3A}\intt e^\phi(\phi-1)+1\,\dx.
\]
To choose $\lambda>0$, we estimate the size of crossing term
\[
|\calC(t)| = \lt|\int_{\T} u\pa_x\phi\,\dx \rt| \le \frac{1}{\rho_-} \int_{\T}  \frac{1}{2}\rho u^2\,\dx + \int_{\T} \frac{1}{2}(\pa_x \phi)^2\,\dx.
\]
Since $0<\rho_- \le 1$, we obtain
\[
\lt(1- \frac{\lambda}{\rho_-}\rt) \calE(t) \le \calE(t) + \lambda \calC(t) \le \lt(1+ \frac{\lambda}{\rho_-}\rt) \calE(t).
\]
We now take 
\[
\lambda:= \min\lt\{ \frac{\nu}{\Lambda}, \frac{\rho_-}{2} \rt\}>0,
\]
then we have
\[
\frac{\textnormal{d}}{\dt} \lt( \calE(t) + \lambda \calC(t) \rt) \le -\kappa \calE(t) \le -\frac{2\kappa}{3}(\calE(t)+ \lambda\calC(t)),
\]
where
\[
\kappa:=\min\lt\{\nu,  \lambda e^{-3A}\rt\} =\lambda e^{-3A}>0.
\]
Thus, we conclude that
\[
\calE(t) \le 2(\calE(t) + \lambda\calC(t)) \le 2(\calE(0) + \lambda\calC(0))e^{-\frac{2\kappa}{3}t} \le 3\calE(0)e^{-\frac{2\kappa}{3}t}.
\]
This completes the proof.
\end{proof}

\subsection{Proof of Theorem \ref{thm:ion}}

%We now prove Theorem \ref{thm:ion}.
%\begin{proof}[Proof of Theorem \ref{thm:ion}]

We first observe from \eqref{eq:phibd} that
\[
e^{-(2\calE(0))^{1/2}} \le e^{\phi(t,x)} \le e^{(2\calE(0))^{1/2}}\quad \forall (t,x)\in\R_+\times \T,
\]
confirming \eqref{eq:cbd}. On the other hand, by applying the mean value theorem and employing \eqref{eq:phi:dcy}, we obtain
\[
 \|e^{\phi} -1 \|_{L^\infty(\T)} \le e^{\|\phi\|_{L^\infty(\T)}}\|\phi\|_{L^\infty(\T)}\le C_* \calE(t)^{1/2}, 
\]
where $C_*=2^{1/2}e^{(2\calE(0))^{1/2}}>0$. We then apply Proposition \ref{prop:1} to obtain
\[
\|e^\phi- 1\|_{L^\infty} \leq C_1e^{-r_1 t}
\]
for some $C_1=C_1(\calE(0))>0$ and $r_1=r_1(\nu,\rho_\pm,M,\calE(0))>0$.
This verifies \eqref{eq:c2}. Combined with uniform bounds \eqref{eq:asm2}, we apply Theorem \ref{thm:general} to obtain the decay estimates of the terms 
\[
\|(\rho-1\,,\,u,\,\partial_x u\,,\,\pa_x\phi,\pa^2_x\phi)\|_{L^{\infty}(\T)},
\]
thereby concluding the proof.
%\end{proof}

%%%%%%%%%%%%%%%%%%%%%%%%%%%%%%%%%%%%%%%%%%%%%%%%%%%

%%%%%%%%%%%%%%%%%%%%%%%%%%%%%%%%%%%%%%%%%%%%%%%%%%%

\end {document}